\documentclass[11pt]{article}

\usepackage{enumitem}
\usepackage{epsfig}
\usepackage[mathscr]{euscript}
\usepackage{csquotes}

\usepackage[linesnumbered, ruled,noend,noline]{algorithm2e}

\usepackage{float}
\usepackage{graphicx}
\usepackage{ stmaryrd }
\usepackage{multirow}
\usepackage{gensymb}
\usepackage{array}
\usepackage{listings}
\usepackage{amsthm}

\renewcommand\thesubsection    {\thesection.\arabic{subsection}}
\usepackage{amssymb}
\usepackage{graphicx,amsmath,calc}

\newtheorem{theorem}{Theorem}[section]

\newtheorem{corollary}[theorem]{Corollary}
\newtheorem{definition}[theorem]{Definition}
\newtheorem{remark}[theorem]{Remark}

\numberwithin{equation}{section}

\title{Intrinsic Information Flow in Structureless NP Search}
\author{Jing-Yuan Wei\thanks{Wide Area Grid-Forming (Shenzhen) Energy Co. Ltd. 
3039 Bao'an North Road, Shenzhen, China. Email: weijingyuan@gmail.com}
}
\date{April 2026} 

\begin{document}
\maketitle

\begin{abstract}
Rather than measuring NP search in terms of Turing-machine time,
we reinterpret witness recovery as an information-acquisition process:
the hidden witness is the sole source of uncertainty,
and identification requires sufficient reduction of this uncertainty
through a rate-limited access interface in the sense of Shannon.

To make this perspective explicit, we analyze an extreme regime,
the \emph{psocid model}, in which the witness is accessible only via
equality probes $[\pi = w^\star]$ under a uniform, structureless prior.
Each probe reveals at most $O(N/2^N)$ bits of mutual information,
so polynomially many probes accumulate only $o(1)$ total information.
By Fano's inequality, reliable recovery requires $\Omega(N)$ bits,
creating a fundamental mismatch between the information
required for recovery and that obtainable through the interface.

The psocid setting isolates a fully symmetric search regime
in which no intermediate computation yields global eliminative leverage,
thereby exposing an intrinsic informational origin of exponential search complexity.
\end{abstract}


\section{Introduction}

At the core of NP search lies a familiar asymmetry:
verifying a proposed witness is easy, while \emph{discovering} it among
exponentially many candidates may be far harder.
Bellare and Goldwasser~\cite{BellareGoldwasser1994} showed that,
under suitable complexity assumptions, there exist languages in NP
for which search does not reduce to decision,
providing evidence that witness recovery may be strictly harder than verification.

As a complementary perspective,
we adopt an information-theoretic viewpoint:
\emph{witness discovery is an information-acquisition process}.
This viewpoint builds on ideas from communication complexity,
where computational difficulty has long been framed in terms of
information transfer since the work of Abelson~\cite{Abelson1978}
and Yao~\cite{Yao1979}.
Yao's minimax principle~\cite{Yao1977} shows that randomized lower bounds
can be obtained via suitable distributional analysis,
an approach central to modern communication complexity.
In Shannon's framework, successful computation requires sufficient
information transfer through a constrained interface.
Our contribution is to internalize this information-theoretic viewpoint
directly within NP witness discovery.

To make this framework concrete, we study an extreme access regime,
the \emph{psocid model}: an instance consists of a library with $2^N$
pages indexed by $\{0,1\}^N$, exactly one of which is marked.
The marked page has index $w^\star \in \{0,1\}^N$.
In each round, $p(N)$ parallel probes - where $p(N)$ is polynomially bounded - 
inspect one page each and receive a single bit indicating whether that page is marked.
Once the marked page is found, the algorithm outputs the $N$-bit index together with
a constant-size certificate, enabling polynomial-time verification.

This separation between fast verification and slow information
acquisition appears in many domains - database
auditing,
scientific search-assay pipelines, and mineral exploration drilling - 
all of which permit rapid local checks while limiting how quickly
the correct candidate can be identified.

A similar separation arises in large-scale infrastructure inspection.
One real-world example occurs in a high-speed rail overhead contact system,
where approximately three million screws are photographed every three months
during scheduled nighttime maintenance windows.
Each photograph is later inspected individually,
with more than twenty inspectors reviewing the images.
Verifying a single screw takes only constant time,
yet locating a rare loose component necessitates scanning
an enormous number of candidates.
When a defective screw is detected, engineers are dispatched
to replace it in the field.

The psocid model isolates this rare-event inspection structure
in its simplest, fully symmetric form. Here, the bottleneck is not computational complexity
but the limited information obtained per inspection:
each image answers only a single local question,
and most inspections return a negative result.

\paragraph{Equality-only access.}
In the psocid model, the witness influences the computation only through
evaluations of 
\[
\mathrm{EQ}(\pi) := [\, \pi = w^\star \,],
\]
which returns $1$ if $\pi$ equals the hidden witness and $0$ otherwise.
We refer to such evaluations as \emph{equality probes}.
Under a uniform prior on $w^\star$,
each probe outcome is a highly biased Bernoulli random variable with entropy
\[
H = h(2^{-N}) = O(N/2^N),
\]
where $h(p) := -p\log p -(1-p)\log(1-p)$ denotes the entropy function.
So each probe conveys only exponentially small mutual information
about the witness.

This raises a fundamental question:  
\emph{Can any polynomial-time algorithm identify $w^\star$ in this access model?}

\paragraph{Main results.}
We answer this question negatively via an information-theoretic barrier.
Fano's inequality implies that recovering an $N$-bit witness with
constant success probability requires $\Omega(N)$ bits of mutual
information.
However, each probe yields at most $O(N/2^N)$ bits of mutual information,
so polynomially many probes can accumulate only $o(1)$ total mutual
information.
This creates a fundamental mismatch between the required and obtainable
information, yielding an information-theoretic impossibility of
polynomial-time recovery under this probe access model.

Our conclusions are specific to the psocid setting under a
uniform, structureless prior.
They do not constitute a general claim about NP search
in the standard Turing-machine model.
The psocid access regime is not intended as a natural or
universal abstraction of NP search.
Rather, it serves to expose an extreme informational bottleneck
and thereby clarify the consequences of viewing witness discovery
as an information-acquisition process.
Whether and how this perspective extends to broader
computational models remains for future investigation.

Section~\ref{sec:npsearch} formalizes the psocid language.
Section~\ref{sec:model} presents the access model.
Section~\ref{sec:abstract-lb} proves the information-theoretic barrier.
Section~\ref{sec:space} derives the time-space tradeoff.
Section~\ref{sec:interpretation} discusses conceptual context.
Section~\ref{sec:conclusion} concludes.

\section{The Psocid Search Problem}
\label{sec:npsearch}

To prepare for the psocid access model introduced in the next section,
we first formalize the associated decision problem
in standard complexity-theoretic terms.

\begin{definition}[Verifier, after~\cite{Sipser2013}]
A \emph{verifier} for a language $A$ is a deterministic algorithm
$V(w,c)$ that takes an instance $w$ and a certificate $c$.
The language is
\[
A = \{\, w \mid \exists\, c \text{ such that } V(w,c) \text{ accepts} \,\}.
\]
If $V$ runs in time polynomial in $\lvert w\rvert$, then
$A$ is \emph{polynomially verifiable}, and hence $A \in \mathbf{NP}$.
\end{definition}

\begin{definition}[Psocid-SAT]
Let $N \ge 1$. Each page in the library is indexed by an
$N$-bit string $w \in \{0,1\}^N$, and exactly one page (if any)
contains a psocid mark. The associated decision
language is
\[
\mathrm{Psocid\mbox{-}SAT}
  =\{\, w \mid \exists\,c \text{ such that } V(w,c)\text{ accepts}\,\},
\]
where $V$ is a deterministic polynomial-time verifier. A certificate has the
form $c=(w^\star,\text{photo})$, consisting of the claimed index $w^\star$ and
an $O(1)$ photo witness certifying the psocid mark. The verifier accepts if and only if
the photo is valid and the page indexed by $w^\star$
is marked in the manner shown in the photo..
\end{definition}

\begin{theorem}
$\mathrm{Psocid\mbox{-}SAT} \in \mathbf{NP}$.
\end{theorem}

\begin{proof}
Given an input index $w$ and a certificate $c=(w^\star,\text{photo})$,
the verifier $V(w,c)$:
\begin{enumerate}[label=\arabic*.,leftmargin=2em]
\item verifies in $O(1)$ time that the \textit{photo} depicts a psocid mark;
\item checks that the page indexed by $w^\star$ contains a psocid mark matching the one shown in the \textit{photo}.
\end{enumerate}
Both steps run in time polynomial in $N=\lvert w\rvert$, and the
certificate has size $O(1)$.
Hence $\mathrm{Psocid\mbox{-}SAT}$ is polynomial-time verifiable
and lies in $\mathbf{NP}$.
\end{proof}

\begin{remark}[Decision versus search]
The decision problem above is trivially in $\mathbf{NP}$,
as verification is local and efficient.
The difficulty studied in this paper lies not in verification,
but in the \emph{search} task of discovering the unique marked
index $w^\star$ under restricted information access.
\end{remark}

\begin{remark}[Parallelism bound $p(N)$]
In the access model of Section~\ref{sec:model}, we allow
$p(N)$ parallel searchers with $p(N)$ polynomially bounded.
Since each searcher records at least one probe outcome per round,
the aggregate workspace scales with $p(N)$.
Restricting $p(N)$ to be polynomial ensures that total resources
remain polynomial.
Allowing superpolynomial parallelism would implicitly permit
superpolynomial information storage per round, moving the model
outside the intended NP regime.
\end{remark}

\section{Two-Stage, One-Way Search-Verification Model}
\label{sec:model}

The psocid search task can be viewed as a communication process
over a noiseless but capacity-limited channel linking three entities:
the instance (which holds the hidden index $w^\star$),
the parallel searchers, and the librarian.
This interpretation follows the communication-complexity viewpoint
of Yao~\cite{Yao1979}.

The interaction proceeds in two one-way stages.
In the \emph{search stage}, information flows from the instance
to the searchers through equality probes,
each revealing a single bit of feedback.
In the \emph{verification stage}, once a candidate is located,
a short certificate (e.g., a constant-size photo)
is transmitted to the librarian,
who outputs a binary decision.
Both stages are one-way and capacity-limited:
the first by the equality-probe interface,
the second by the bounded reporting bandwidth required
to transmit the certificate to the verifier.
Verification itself is computationally efficient.

\paragraph{Setup.}
Let $n := 2^N$ denote the number of pages and
let $w^\star \in \{0,1\}^N$ be the unique index of the marked page.
The goal is to discover $w^\star$ hiring $p(N)$ parallel searchers,
where $p(N)$ is polynomially bounded.

In round $t$, searcher $j \in [p(N)]$ probes a candidate
$\pi(t,j)\in\{0,1\}^N$ and receives
\[
Y_{t,j} := [\,\pi(t,j)=w^\star\,] \in \{0,1\}.
\]
The process halts at the first success,
\[
T_{\mathrm{search}} := \min\{\, t : \exists j \text{ with } Y_{t,j}=1 \,\}.
\]
Under a public probing schedule without replacement,
the pair $(T_{\mathrm{search}},J)$ uniquely determines $w^\star=\pi(T_{\mathrm{search}},J)$.
Throughout this section, $w^\star$ is drawn uniformly from $\{0,1\}^N$.


\paragraph{Information in a single probe.}

Let $p := 1/n = 2^{-N}$.
For a single probe, the outcome $Y$ satisfies
\[
\Pr(Y=1)=p, \qquad \Pr(Y=0)=1-p.
\]
Define the binary entropy function
\[
h(p) := -p\log p - (1-p)\log(1-p).
\]
Throughout this section, $\log(\cdot)$ denotes the base-2 logarithm.
Then $Y \sim \mathrm{Bernoulli}(p)$ and its Shannon entropy (in bits) is
\[
H(Y)=h(p).
\]

The mutual information between $w^\star$ and a probe outcome $Y$ is 
(Cover-Thomas~\cite{CoverThomas2006}), 
\[
I(w^\star;Y)
=
H(Y)-H(Y\mid w^\star)
\le H(Y)
=
h(p).
\]


\paragraph{Search-stage information accumulation.}

$T_{\mathrm{search}}$ is the (random) round of the first hit, and define
$q := T_{\mathrm{search}}\cdot p(N)$ as the (random) number of probes issued
up to and including that round.
To apply the chain rule, we flatten the $p(N)$ parallel outcomes in each round
into a single sequence of $q$ scalar bits
\[
\mathcal{F}_q := (y_1,\dots,y_q),
\]
where $\{y_k\}_{k \ge 1}$ enumerates the coordinates $Y_{t,j}$ in lexicographic order of $(t,j)$.
This is purely a notational device: each outcome $Y_{t,j}$ up to and including round
$T_{\mathrm{search}}$ appears exactly once in $\mathcal{F}_q$.

After $k-1$ misses, the posterior over candidates
remains uniform over $n-(k-1)$ indices.
Hence
\[
\Pr(y_k=1\mid y_{<k})
=
\frac{1}{n-(k-1)}  ~~
\Longrightarrow ~~ H(y_k\mid y_{<k})
=
h\!\left(\frac{1}{n-(k-1)}\right)
\]
so
\[
I(w^\star;y_k\mid y_{<k})
\le H(y_k\mid y_{<k}) =
h\!\left(\frac{1}{n-(k-1)}\right).\]

By the chain rule,
\begin{equation}
\label{eq:chainrule}
I(w^\star;\mathcal{F}_q) =\sum_{k=1}^q I(w^\star;y_k\mid y_{<k})
\le 
\sum_{k=1}^q
h\!\left(\frac{1}{n-(k-1)}\right).
\end{equation}


\paragraph{Information required for reliable recovery.}

The previous calculation quantifies exactly how much mutual
information can be accumulated after $q$ probes.
We now compare this obtainable information with the amount
required to identify $w^\star$ reliably.

To recover $w^\star$ with error probability at most
$\varepsilon<1/3$, the probe transcript must reduce
this uncertainty by a constant fraction.
Formally, Fano's inequality (Cover-Thomas~\cite{CoverThomas2006}) gives
\[
H(w^\star \mid \mathcal{F}_q)
   \;\le\;
   h(\varepsilon)
   + \varepsilon \log(n-1),
\]
where $n=2^N$ is the size of the hypothesis space.
Since
$I(w^\star;\mathcal{F}_q)=H(w^\star)-H(w^\star\mid\mathcal{F}_q)$,
\[
I(w^\star;\mathcal{F}_q)
   \;\ge\;
   H(w^\star)
   - h(\varepsilon)
   - \varepsilon\log(n-1).
\]

Because $w^\star$ is uniform on $\{0,1\}^N$, it has entropy
\[
H(w^\star)
   = \sum_{i=1}^{n} \log(n)/n
   = \log(n).
\]
Thus a uniformly random $N$-bit index contains exactly $\log(n)=N$ bits of
uncertainty. Substituting $H(w^\star)=\log(n)$,
\begin{equation}
\label{eq:fano}
\begin{aligned}
& I(w^\star;\mathcal{F}_q) &
   \;\ge\; &
   \log(n) - h(\varepsilon)
       - \varepsilon\log(n-1)\\
& &\; =  \;&(1-\varepsilon)\log(n)- h(\varepsilon) + o(1),
\end{aligned}
\end{equation}
where $o(1) \to 0$ as $n \to \infty$.

Combining \eqref{eq:chainrule} and \eqref{eq:fano}, to recover $w^\star$ with error probability at most $\varepsilon<1/3$,
it is necessary that
\begin{equation}
\label{eq:qbound}
\sum_{k=1}^q h\!\left(\frac{1}{n-(k-1)}\right)
\;\ge\;
(1-\varepsilon)\log n
- h(\varepsilon) + o(1).
\end{equation}

Appendix~A (see \eqref{eq:q-theta}) shows that this inequality forces
\[
q \,\ge\, \bigl(1-e^{-(1-\varepsilon)+o(1)}\bigr)n= \Theta(n) = \Theta(2^N).
\]
Since $q=T_{\mathrm{search}}\cdot p(N)$,
\begin{equation}
\label{eq:Tsearch-clean}
T_{\mathrm{search}}
=
\Omega\!\left(\frac{2^N}{p(N)}\right).
\end{equation}


\paragraph{Verification stage.}

Once the marked page has been located,
the searcher must transmit its $N$-bit index
together with a constant-size photo to the librarian.

Communication in this stage is one-way and bandwidth-limited.
In each round, at most $p(N)$ bits can be transmitted:
each of the $p(N)$ parallel searchers can send at most
a constant-size message per round.
Let $A_t$ denote the $t$-th payload and
$\mathcal{G}_r = (A_1,\dots,A_r)$
the verification transcript.
Since each $A_t$ is a binary string of length at most $p(N)$,
\(
|A_t| \le p(N).
\)
Thus the total transcript length satisfies
\[
|\mathcal{G}_r|
=
\sum_{t=1}^r |A_t|
\le
r\,p(N).
\]

To verify the witness, the librarian must recover the
$N$-bit index $w^\star$.
Thus the verification transcript must convey $\Theta(N)$ bits
of information.
Consequently,
\(
r\,p(N) \;\ge\; N,
\)
and 
\(
r \;\ge\; N/p(N).
\)

Therefore the verification stage requires
\begin{equation}
\label{eq:Tverify-clean}
T_{\mathrm{verify}}
=
\Omega\!\left(\frac{N}{p(N)}\right).
\end{equation}


\paragraph{Total time.}

Since $T=T_{\mathrm{search}}+T_{\mathrm{verify}}$,
combining \eqref{eq:Tsearch-clean} and \eqref{eq:Tverify-clean} yields
\begin{equation}
\label{eq:T-total}
T = \Omega\!\left(\frac{2^N}{p(N)}\right)
+ \Omega\!\left(\frac{N}{p(N)}\right) = 
\Omega\!\left(\frac{2^N}{p(N)}\right).
\end{equation}


\begin{remark}[Expected search time vs.\ information requirement]
Under probing without replacement and a uniform prior on $w^\star$,
the position of the marked page in any fixed probing order is uniform
over $\{1,\dots,n\}$.
Hence the expected number of probes until the first hit is
\[
\mathbb{E}[Q] = \frac{n+1}{2}.
\]
This average-case statement concerns the stopping time of search.

Our information-theoretic bound is of a different nature.
It does not analyze the expected location of the first hit.
Instead, it quantifies how many probes are required before the probe
transcript can contain sufficient mutual information to identify
$w^\star$ reliably.

In particular, the entropy bound for zero-error recovery implies
\[
q \ge (1-e^{-1})\,n,
\]
where $1-e^{-1} \approx 0.632$.
Thus a linear number of probes is necessary before the transcript can
determine the witness.
Equivalently, if $q < (1-e^{-1})\,n$, the probe transcript cannot carry
sufficient mutual information to determine $w^\star$ with zero error
under the uniform prior.
\end{remark}


\begin{remark}[Information interface and scope]
The lower bound is model-dependent: it applies to algorithms that access
the witness solely through equality probes.
The psocid framework restricts only this information interface; internal
computation remains unrestricted.
Under the uniform prior, each probe outcome
$Y$ has entropy $h(1/n)=O((\log n)/n)$, independent of scheduling.
Thus all algorithms receive the same exponentially vanishing
per-probe information, and since the internal state of the algorithm
is a function of the probe transcript, the mutual information with
$w^\star$ cannot exceed that contained in the transcript.
\end{remark}

\section{An Information-Theoretic Barrier to Polynomial-Time Psocid Search}
\label{sec:abstract-lb}

In the psocid model, the witness \(w^\star \in \{0,1\}^N\) is accessible
only through equality probes of the form \([\,\pi = w^\star\,]\).
Section~\ref{sec:model} shows that each probe reveals at most
\(O(N/2^N)\) bits of mutual information about \(w^\star\).
Thus the probe interface defines an information channel with
exponentially vanishing per-use capacity.

We show that polynomially many uses of this channel cannot
accumulate sufficient information to recover an \(N\)-bit witness
with non-negligible success probability.

\paragraph{Standing premise.}
Since \(w^\star\) is drawn uniformly from
\(\{0,1\}^N\), 
all statements in this section are conditioned on this uniform, structureless prior.

An arbitrary (adaptive, randomized, or parallel) algorithm performs
\(q\le poly(N)\) probes for some polynomial \(poly(\cdot)\), producing the
transcript
\[
\mathcal{F}_q = (y_1,\dots,y_q), \qquad y_k \in \{0,1\}.
\]
After observing \(\mathcal{F}_q\), the algorithm outputs
\[
\hat w := A(\mathcal{F}_q) \in \{0,1\}^N.
\]
We say the algorithm succeeds with non-negligible probability if
\(\Pr[\hat w = w^\star] \ge \delta\) for some fixed constant
\(\delta>0\), independent of \(N\).


\paragraph{Step 1: Successful recovery requires linear information.}

Let \(P_e := \Pr[\hat w \neq w^\star] \le 1-\delta\).
By Fano's inequality,
\[
H(w^\star \mid \hat w)
\le
h(P_e) + P_e \log(2^N-1).
\]
Since \(w^\star\) is uniform on \(\{0,1\}^N\),
\(H(w^\star)=N\), and hence
\[
I(w^\star;\hat w)
=
H(w^\star)-H(w^\star\mid \hat w)
\ge
N - h(P_e) - P_e \log(2^N-1).
\]
Using \(\log(2^N-1)=N-o(1)\) and \(P_e \le 1-\delta\),
\[
I(w^\star;\hat w)
\ge
\delta N - h(1-\delta) + o(1)
\ge
cN,
\]
for some constant \(c>0\) depending only on \(\delta\).

Thus any recovery procedure with constant success probability
must acquire linear mutual information about \(w^\star\).

Since the algorithm's output $\hat w$ is computed solely from the full
$q$-probe transcript $\mathcal{F}_q$ (including its internal randomness),
we have
\[
\Pr(\hat w \mid \mathcal{F}_q, w^\star)
  = \Pr(\hat w \mid \mathcal{F}_q),
\]
so that
\(
w^\star \;\to\; \mathcal{F}_q \;\to\; \hat w
\)
forms a Markov chain.
By the data-processing inequality~\cite{CoverThomas2006},
\begin{equation}
\label{eq:dpi}
I(w^\star ; \mathcal{F}_q)
   \;\ge\; I(w^\star ; \hat w)
   \;\ge\; cN,
\end{equation}
for sufficiently large \(N\).


\paragraph{Step 2: Polynomially many probes carry vanishing information.}

By the chain rule (Section~\ref{sec:model}),
\[
I(w^\star;\mathcal{F}_q)
=
\sum_{k=1}^q
I(w^\star;y_k \mid y_{<k})
\,\le \,
\sum_{k=1}^q
h\!\left(\frac{1}{n-(k-1)}\right).
\]
Appendix~A (see~\eqref{eq:poly-o1}) shows that if \(q \le \mathrm{poly}(N)\),
then the right-hand side of the inequality above equals \(o(1)\).
Consequently,
\begin{equation}
\label{eq:cap-bound-clean}
I(w^\star;\mathcal{F}_q)
\;=\;
o(1).
\end{equation}

Thus polynomially many probes accumulate only
vanishing total mutual information.


\paragraph{Step 3: Incompatibility.}

Equations~\eqref{eq:dpi} and~\eqref{eq:cap-bound-clean}
cannot simultaneously hold for sufficiently large \(N\):
the former requires
\(I(w^\star;\mathcal{F}_q) \ge cN\),
whereas the latter carries
\(I(w^\star;\mathcal{F}_q)= o(1)\)
whenever \(q \le \mathrm{poly}(N)\).

This contradiction shows that polynomially many probes
cannot identify a uniformly random \(N\)-bit witness
with constant success probability.


\begin{theorem}
\label{thm:capacity}
In the psocid model under a uniform prior,
for every polynomial \(poly(\cdot)\) and every constant \(\delta>0\),
no algorithm making at most \(poly(N)\) probes can recover
\(w^\star \in \{0,1\}^N\) with success probability at least \(\delta\).
More precisely, any polynomial-length transcript carries only
\(o(1)\) mutual information about \(w^\star\).
\end{theorem}

\begin{corollary}[Polynomial-time recovery is impossible]
\label{cor:polytime}
In the psocid model, each probe requires at least one computation step,
so any polynomial-time algorithm can make at most polynomially many
probes. By Theorem~\ref{thm:capacity},
no polynomial-time algorithm can recover \(w^\star\) with success
probability bounded away from zero.
\end{corollary}

\begin{corollary}[Worst-case lower bound via Yao~\cite{Yao1977}]
\label{cor:yao}
Since Theorem~\ref{thm:capacity} establishes an average-case lower bound
under the uniform prior, Yao's minimax principle implies that for every
randomized algorithm using polynomially many probes, there exists an
input on which the algorithm cannot recover \(w^\star\) with constant
success probability.
\end{corollary}

Thus the obstruction is purely information-theoretic:
the probe channel cannot transmit sufficient information to identify
the witness, regardless of internal computation, adaptivity,
or parallelism.

\section{Time-Space Tradeoff in the Psocid Model}
\label{sec:space}

Beyond the vanishing per-probe information rate established in
Section~\ref{sec:model}, the psocid framework exhibits a natural
time-space tradeoff.

Suppose the $p(N)$ parallel searchers collectively use at most
$S$ bits of workspace throughout the computation.
The workspace $S$ bounds the amount of information that can be stored
at any given time, but it does not increase the rate at which new
information about $w^\star$ can be acquired through probes.

Since at most $p(N)$ probes occur per round, the natural regime is
$S = \Theta(p(N))$, meaning that each searcher maintains only
constant or polylogarithmic local memory and total workspace scales
linearly with the degree of parallelism.

Let $T$ denote the number of rounds required for successful discovery.
From \eqref{eq:T-total},
\(
T \;=\; \Omega(2^N/p(N)).
\)
Multiplying by $S=\Theta(p(N))$ yields
\[
T S
\;=\;
\Omega\!\left(\frac{2^N}{p(N)} \cdot p(N)\right)
\;=\;
\Omega(2^N).
\]

\paragraph{Interpretation.}
The tradeoff does not stem from limits on internal computation:
the algorithm may perform arbitrary processing within its workspace.
The bottleneck lies entirely in the probe interface,
through which all information about $w^\star$ must pass.
Each equality probe conveys only $O(N/2^N)$ bits of mutual information,
so increasing workspace does not increase the effective information rate.

Consequently, even with polynomial space and polynomially many
parallel searchers, exponential time is unavoidable under
equality-only access and a uniform prior.
Equivalently, the psocid model obeys the fundamental constraint
\(
TS = \Omega(2^N),
\)
which reflects an information-theoretic limitation.

\section{Conceptual Lineage}
\label{sec:interpretation}

Communication complexity
(Abelson~\cite{Abelson1978}, Yao~\cite{Yao1979})
established that computational hardness may arise from
constraints on information flow rather than from limits on local processing.
Yao's minimax principle~\cite{Yao1977} further showed that
distributional analysis can yield lower bounds that hold uniformly over randomized algorithms.
Together, these developments highlight the central role of
information constraints in understanding algorithmic difficulty.

We internalize this perspective within NP witness recovery,
viewing the task as a rate-limited information-acquisition process.

To this end, we study the psocid model,
a minimal setting in which structural features are removed.
Under a uniform prior on the hidden witness,
the equality-probe interface can be interpreted as a Shannon channel
with vanishing capacity.
The resulting barrier arises from an information-rate mismatch:
reliable recovery requires $\Theta(N)$ bits of mutual information,
whereas polynomially many probes yield only $o(1)$ bits.

This stands in contrast to structured NP settings,
where computational steps can prune large portions of the search space:
a clause eliminates all assignments that falsify its literals,
a cutting plane removes an entire feasible subregion,
and a feasibility certificate may rule out exponentially many possibilities.
Such structure provides global eliminative leverage.

In structured NP problems such as 3-SAT,
each pruning step reduces uncertainty about the solution,
yet the resulting information gain is difficult to quantify:
overlapping constraints and correlated eliminations prevent a clean,
additive characterization in terms of mutual information.
Empirically, structured search often yields significant local information early on,
while the marginal information gain may diminish as the search progresses,
and remain limited in certain critical instances,
leaving a residual search space in which candidates are comparatively difficult to distinguish.

The psocid interface deliberately eliminates this leverage:
each probe excludes only a single candidate,
while preserving symmetry among the remaining possibilities.
This minimality ensures that the difficulty of search is governed entirely
by constraints on information acquisition, yielding a fundamental mismatch
between required and obtainable information.

\medskip
\noindent
\textbf{Relation to query and decision-tree models.}
Query and decision-tree lower bounds typically rely on adversary arguments
and combinatorial techniques that limit how effectively queries can distinguish
among possible inputs.
Such bounds are inherently tied to the underlying query model,
as they quantify the distinguishing power of queries within that model,
and may be circumvented in more powerful computational models such as parallel circuits.
Our work does not seek to strengthen such bounds.
Instead, it recasts unstructured NP witness recovery within an
explicit Shannon-theoretic framework,
deriving impossibility from the accumulation of mutual information
under a fixed uniform prior.
The conceptual contribution is to isolate a regime in which the access interface itself
has vanishing information rate, yielding limitations that cannot be overcome
by additional computation or parallelism.

\section{Concluding Remarks}
\label{sec:conclusion}

This work advances a unifying perspective:
NP witness discovery is viewed as an
\emph{information-acquisition process}.

From this perspective, an algorithm $A$ solving an instance $x$
incurs cost $c(A,x)$ and accumulates a bounded amount
of mutual information about the hidden solution.
Exact recovery requires that the accumulated information
match the entropy of the solution.
Accordingly, the cost required for recovery is governed by the
intrinsic information rate of the access interface:
higher-rate interfaces permit the requisite information
to be acquired with fewer probes (lower cost),
whereas lower-rate interfaces necessitate more.
As a manifestation of this principle,
the psocid model isolates an extreme access regime in which
this intrinsic rate vanishes exponentially.
Consequently, exponentially many probes are required
for reliable witness recovery.

Within NP search, it is natural to ask whether other problems - perhaps in certain critical instances or in later stages of the search - admit
similarly low-capacity access regimes.
Our analysis of the psocid model suggests that exponential search
behavior may arise whenever the intrinsic information flow of the
interface yields vanishing information gain per interaction.

\appendix
\renewcommand{\theequation}{A.\arabic{equation}}
\renewcommand{\thesubsection}{A.\arabic{subsection}}
\setcounter{equation}{0}
\setcounter{subsection}{0}

\section*{Appendix: Asymptotics of Accumulated Mutual Information}

Throughout this appendix, we use natural logarithms.
Changing the logarithm base multiplies all entropy values by the
constant factor $\ln 2$, and therefore affects only constant factors
in the bounds.

\subsection{Accumulation of Mutual Information}

Expanding the right-hand side of \eqref{eq:chainrule} using natural
logarithms, we obtain
\[
\begin{aligned}
& (\ln 2)\sum_{k=1}^q h\!\left(\frac{1}{n-(k-1)}\right) \\
= ~~&
\sum_{k=1}^{q}
\left(
\frac{\ln (n-(k-1))}{n-(k-1)}
-
\Bigl(1-\frac{1}{n-(k-1)}\Bigr)
\ln\Bigl(1-\frac{1}{n-(k-1)}\Bigr)
\right).
\end{aligned}
\]

Let $m:= n-(k-1)$. Since $m$ decreases from $n$ to $n-(q-1)$, 
\begin{equation}\label{eq:mutual-split}
\begin{aligned}
& (\ln 2)\sum_{k=1}^q h\!\left(\frac{1}{n-(k-1)}\right)   = 
(\ln 2)\sum_{m=n}^{m=n-(q-1)} h\!\left(\frac{1}{m}\right) \\
=
&\sum_{m=n-(q-1)}^{n}
\left(
\frac{\ln m}{m}
-
\Bigl(1-\frac{1}{m}\Bigr)
\ln\Bigl(1-\frac{1}{m}\Bigr)
\right) \\
=
&\sum_{m=2}^{n}
\left(
\frac{\ln m}{m}
-
\Bigl(1-\frac{1}{m}\Bigr)
\ln\Bigl(1-\frac{1}{m}\Bigr)
\right) \\
&-
\sum_{m=2}^{n-q}
\left(
\frac{\ln m}{m}
-
\Bigl(1-\frac{1}{m}\Bigr)
\ln\Bigl(1-\frac{1}{m}\Bigr)
\right).
\end{aligned}
\end{equation}

We use the asymptotics
\begin{equation}\label{eq:asymp}
\begin{aligned}
&~~~\sum_{m=2}^{\bar m} \frac{\ln m}{m}
=
\frac12(\ln \bar m)^2 + \gamma_1 + o(1),~~and\\
&-\sum_{m=2}^{\bar m}
\Bigl(1-\frac{1}{m}\Bigr)
\ln\Bigl(1-\frac{1}{m}\Bigr)
=
\ln \bar m + \gamma' + o(1),
\end{aligned}
\end{equation}
where $\gamma_1$ is the first Stieltjes constant
and $\gamma' \approx -0.7885305659$ is an explicit absolute constant.

Substituting \eqref{eq:asymp} into \eqref{eq:mutual-split},
\[\begin{aligned}
(\ln 2)\sum_{k=1}^q h\!\left(\frac{1}{m}\right)
&=
\frac12(\ln n)^2
-\frac12(\ln(n-q))^2 \\
&\quad
+(\ln n-\ln(n-q))
+o(1).
\end{aligned}\]

Rewriting in terms of
\[
x := \ln\!\Bigl(\frac{n}{n-q}\Bigr),
\]
we obtain
\begin{equation}\label{eq:final-form}
(\ln 2)\sum_{k=1}^q h\!\left(\frac{1}{m}\right)
=
(\ln n+1)x - \frac12 x^2 + o(1).
\end{equation}

\subsection{Solving for the Threshold of $q$}

From the Fano lower bound \eqref{eq:qbound},
\[
(\ln 2)\sum_{k=1}^q h\!\left(\frac{1}{m}\right)
\;\ge\;
(1-\varepsilon)\ln n - (\ln 2)h(\varepsilon)
\]
for sufficiently large $n$.

Combining with \eqref{eq:final-form} yields
\[(\ln n+1)x - \frac12 x^2
\;\ge\;
(1-\varepsilon)\ln n - (\ln 2)h(\varepsilon).\]

Rearranging,
\[
x^2 - 2(\ln n+1)x
+2(1-\varepsilon)\ln n - 2(\ln 2)h(\varepsilon)
\;\le\; 0.
\]

Solving the quadratic inequality gives
\begin{equation}\label{eq:x-bound}
x
\ge
(\ln n+1)
-
\sqrt{(\ln n+1)^2
-2(1-\varepsilon)\ln n
+2(\ln 2)h(\varepsilon)}.
\end{equation}

Write
\[
\begin{aligned}
&\sqrt{(\ln n+1)^2-2(1-\varepsilon)\ln n+2(\ln 2)h(\varepsilon)} \\
= ~&
(\ln n+1)
\sqrt{1 - \frac{2(1-\varepsilon)\ln n-2(\ln 2)h(\varepsilon)}{(\ln n+1)^2}}.
\end{aligned}
\]
For large $n$, using the Taylor expansion $\sqrt{1-u}=1-\frac{u}{2}+O(u^2)$ as $u\to 0$,
we obtain
\[
\sqrt{(\ln n+1)^2-2(1-\varepsilon)\ln n+2(\ln 2)h(\varepsilon)}
=
(\ln n+1)
-
(1-\varepsilon)
+
o(1),
\]
where $o(1)\to 0$ as $ n\to\infty$.

Substituting into \eqref{eq:x-bound} yields
\(
x \ge (1-\varepsilon) - o(1).
\)

Returning to $q$,
\[
x=-\ln\!\Bigl(1-\frac{q}{n}\Bigr)
\ge (1-\varepsilon)-o(1),
\]
so
\[
\frac{q}{n}
\ge
1-e^{-(1-\varepsilon)+o(1)}.
\]

Hence
\begin{equation}\label{eq:q-theta}
q \ge \bigl(1-e^{-(1-\varepsilon)+o(1)}\bigr)n
= \Theta(n).
\end{equation}

\subsection{Polynomially Many Probes}

Assume now $q=poly(\ln n)$.
Then $q=o(n)$ and $q/n\to 0$.
Using $\ln(1-u)=-u+O(u^2)$ with $u=q/n$,
\[
x= -\ln\!\Bigl(1-\frac{q}{n}\Bigr)
=
\frac{q}{n}
+O\!\left(\frac{q^2}{n^2}\right).
\]

Substituting into \eqref{eq:final-form},
\[
(\ln 2)\sum_{k=1}^q h\!\left(\frac{1}{m}\right)
=
(\ln n+1)\frac{q}{n}
+O\!\left(\frac{(\ln n)q^2}{n^2}\right).
\]

Since $q=poly(\ln n)$,
\[
(\ln n+1)\frac{q}{n}\to 0,
\qquad
\frac{(\ln n)q^2}{n^2}\to 0,
\]
and therefore
\begin{equation}\label{eq:poly-o1}
(\ln 2)\sum_{k=1}^q h\!\left(\frac{1}{m}\right) = o(1)\quad\Longrightarrow\quad \sum_{k=1}^q h\!\left(\frac{1}{n-(k-1)}\right)
= o(1).
\end{equation}


\end{document}